\documentclass[10pt]{article}

\usepackage{todonotes}

\usepackage{amsmath,amsthm,amssymb, fullpage, graphicx,psfrag}
\usepackage{hyperref}
\usepackage{lineno}
\newtheorem{thm}{Theorem}[section]
\newtheorem{prop}[thm]{Proposition}

\newtheorem{conj}[thm]{Conjecture}
\newtheorem{ex}[thm]{Example}

\theoremstyle{definition}
\newtheorem*{defn}{Definition}
\newtheorem{case}{Case}
\newtheorem{subcase}{Case}[case]
\usepackage{color}
\definecolor{red}{rgb}{.6,0,0}

\definecolor{blu}{rgb}{0,0,1}

\newcommand{\lp}{\left(}
\newcommand{\rp}{\right)}

\usepackage{tikz}
\usepackage{multicol}

\newcommand{\bit}{\begin{itemize}}
\newcommand{\eit}{\end{itemize}}
\newcommand{\ben}{\begin{enumerate}}
\newcommand{\een}{\end{enumerate}}
\newcommand{\beq}{\begin{equation}}
\newcommand{\eeq}{\end{equation}}
\newcommand{\bea}{\begin{eqnarray*}}
\newcommand{\eea}{\end{eqnarray*}}
\newcommand{\bpf}{\begin{proof}}
\newcommand{\epf}{\end{proof}}

\mathchardef\mhyphen="2D
\newcommand\mrcr{\max\mhyphen\overline{\textnormal{cr}}}

\title{A Note on the Maximum Rectilinear Crossing Number of Spiders}

\author{ 
Joshua Fallon \thanks{Louisiana State University, email: \href{mailto:jfallo3@lsu.edu}{jfallo3@math.lsu.edu}}
\and Kirsten Hogenson \thanks{Colorado College, email: \href{mailto:khogenson@coloradocollege.edu}{khogenson@coloradocollege.edu}}
\and Lauren Keough \thanks{Grand Valley State University, email: \href{mailto:keoulaur@gvsu.edu}{keoulaur@gvsu.edu}}
\and Mario Lomel\'{\i} \thanks{Universidad Aut\'onoma de San Luis Potos\'{\i}, email: \href{mailto:lomeli@ifisica.uaslp.mx}{lomeli@ifisica.uaslp.mx}}
\and Marcus Schaefer \thanks{DePaul University, email: \href{mailto:MSchaefer@cdm.depaul.edu}{MSchaefer@cdm.depaul.edu}}
\and Pablo Sober\'on \thanks{Baruch College, CUNY email: \href{mailto:pablo.soberon-bravo@baruch.cuny.edu}{pablo.soberon-bravo@baruch.cuny.edu}} }

\tikzset{
insep/.style={inner sep=2pt, outer sep=0pt, circle,fill}, 
free/.style={inner sep=2pt, outer sep=0pt, circle,fill=gray,draw}, 
extra/.style={inner sep=2pt, outer sep=0pt, circle,fill=white,draw}, 
}

\begin{document}

\maketitle


\begin{abstract}
	The maximum rectilinear crossing number of a graph $G$ is the maximum number of crossings in a good straight-line drawing of $G$ in the plane.  In a good drawing any two edges intersect in at most one point (counting endpoints), no three edges have an interior point in common, and edges do not contain vertices in their interior.  A spider is a subdivision of $K_{1,k}$.  We provide both upper and lower bounds for the maximum rectilinear crossing number of spiders.  While there are not many results on the maximum rectilinear crossing numbers of infinite families of graphs, our methods can be used to find the exact maximum rectilinear crossing number of $K_{1,k}$ where each edge is subdivided exactly once.  This is a first step towards calculating the maximum rectilinear crossing number of arbitrary trees.
\end{abstract}

\section{Introduction}
\label{sec:intro}

 Planar graphs, which are graphs that can be embedded in the plane such that no two edges intersect except at their endpoints, have long been of interest to the mathematical community. Kuratowski's theorem from 1930 gives us a characterization of which graphs are planar \cite{Kuratowski}.  For graphs that are not planar, it is natural to wonder which drawings in the plane have the fewest crossings.  Tur\'{a}n did just this during his time in a labor camp during World War II \cite{Turan}.

\begin{defn} 
The \emph{crossing number} of a graph $G$ is the minimum number of crossings over all drawings of $G$.
\end{defn}

The crossing number for $K_{m,n}$, originally investigated by Tur\'{a}n, is still unknown.  The conjecture for the optimal bound is known as the Zarankiewicz conjecture \cite{Zarankiewicz}. Similarly, finding the crossing number of $K_n$ is open and the conjecture for the optimal bound is known as Hill's conjecture \cite{HillsConjecture}.  The problem of finding the crossing number of a given graph is NP-complete \cite{CrossingNumberNPComplete}.  For some recent problems on crossing numbers see Chapter 13 of \cite{SzekelyBook}.

There are many variants of the crossing number parameter; for a detailed survey see \cite{Schaefer}.  One such variant is the maximum crossing number which aims at maximizing the number of crossings. For this to make sense, we must insist that we have a good drawing. In a \emph{good} drawing, no edge can cross itself, each pair of edges have at most one point in common (including endpoints), and no more than two edges can cross at any point. 

\begin{defn}
The \emph{maximum crossing number} of a graph $G$  is the maximum number of crossings in a good drawing of $G$ in the plane. We denote the maximum crossing number of $G$ as $\max\mhyphen\text{cr}(G)$. 
\end{defn}

For the remainder of this paper, all drawings will be good drawings in the plane.
By definition a graph $G$ with $m$ edges can have at most $\binom{m}{2}$ crossings in a good drawing. Since incident edges cannot cross each other in a good drawing, the actual number of crossings will be even smaller in general. This observation gives us a natural upper bound   to the maximum crossing number of a graph.
\begin{defn}
 The \emph{thrackle bound} for a graph $G$ with edge set $E(G)$ and vertex set $V(G)$ is
	\[\vartheta(G) = \binom{|E(G)|}{2} - \sum_{u\in V(G)}\binom{\deg(u)}{2} \]
\end{defn}

The thrackle bound is an upper bound for the maximum crossing number \cite{thrackleUB}, as it counts the number of crossings we would have if every pair of non-incident edges crossed.  
A graph is called \emph{thrackleable} if it has a drawing which achieves the thrackle bound. 

Among the most difficult problems concerning maximum crossing numbers is Conway's thrackle conjecture.  Conway conjectures that a thrackleable graph on $n$ vertices cannot have more than $n$ edges. As of 2017, Conway is offering \$1000 for a proof \cite{Conway1000}, which is up from his initial bounty of ten shillings and six pence offered in 1969 \cite{Woodall}.   Much work has been done on this conjecture \cite{Conway1.428n, Conway2n}. At the time of this note's writing, the most recent result posted on arXiv shows that any thrackleable graph on $n$ vertices has at most $1.3984n$ edges \cite{Conway1.3984}.  Conway's thrackle conjecture is true for the rectilinear drawings \cite{Woodall}.

Another interesting open problem concerning maximum crossing numbers is the monotonicity conjecture. For most variations of crossing numbers, monotonicity holds.  That is, if $H$ is a subgraph of $G$ then $\phi(H)$ is at most $\phi(G)$ where $\phi$ is the crossing number variant. The most striking exception is the maximum crossing number, for which monotonicity is an open problem even for induced subgraphs \cite{Schaefer}. Monotonicity is known to hold for rectilinear drawings, a result due to Ringeisen, Stueckle, and Piazza \cite{Monotonicity}.

\begin{defn}
A \emph{rectilinear drawing} of a graph $G$ is a drawing in which all edges are straight line segments. The \emph{maximum rectilinear crossing number} of a graph $G$, $\mrcr(G)$, is the maximum number of crossings in any rectilinear drawing of $G$.
\end{defn}

The maximum rectilinear crossing number has been rediscovered several times under different names, including obfuscation complexity \cite{Obfuscation}.

Since $\mrcr(G) \leq \max\mhyphen\text{cr}(G)$, the thrackle bound is also an upper bound for the maximum rectilinear crossing number. Conway raised the problem of classifying all thrackleable graphs.  In \cite{Woodall}, Woodall gives a complete classification assuming that the thrackle conjecture is true.  He also considers the rectilinear case.  In Woodall's classification of graphs with rectilinear drawings that obtain the thrackle bound, he finds the exact maximum rectilinear crossing number for a particular spider.  A \emph{spider} is a subdivision of $K_{1,k}$.  Each path radiating out from the degree $k$ vertex of the spider is a \emph{leg}.  In general, the legs of the spider may be of different lengths. For the spider with $k$ legs in which each leg has length $2$ we write $S_k^2$.  The graph $S_3^2$ is typically called $T_2$.

\begin{thm}[Woodall~\cite{Woodall}]\label{thm:S333}
The graph $S_3^2$ does not have a rectilinear drawing that obtains the thrackle bound.  In fact, $\mrcr(S_3^2) = 8 = \vartheta(S_3^2)-1$.
\end{thm} 

This result contrasts with the maximum crossing number case, since all trees are thrackleable. Our long-term goal is to better understand the maximum rectilinear crossing numbers of trees, Theorem~\ref{thm:S333} shows that this will require new ideas right from the start. In the current paper we focus on spider graphs, and use Theorem \ref{thm:S333} to compute an upper bound on the maximum rectilinear crossing number of spider graphs.  In Section \ref{sec:LB} we provide an algorithm that gives a lower bound on the maximum rectilinear crossing number of spiders with at least $3$ legs.  We conjecture that this algorithm gives the maximum rectilinear crossing number for spiders.  In Section \ref{sec:UB} we use Mantel's theorem to give an upper bound on $\mrcr(S)$ for any spider $S$.  Letting $S_k^2$ be the spider in which each path is of length two, Woodall proves that $\mrcr(S_3^2)=\vartheta(S_3^2)-1$.  In Theorem~\ref{thm:Len2} We extend this result by showing that $\mrcr(S_k^2) = \vartheta(S_k^2) - \binom{k}{2} + \left\lfloor\frac{k^2}{4}\right\rfloor$.

 For additional information regarding graph theory definitions and notation, we refer the reader to \cite{West}.
 

\subsection*{Previous Results}

Finding the maximum rectilinear crossing number for a given graph is  NP-hard \cite{mrcrNPHard}, so it is not surprising that there are only a few exact results for the maximum rectilinear crossing number of infinite families of graphs; solved cases include complete $k$-partite graphs \cite{H76}, cycles \cite{S23}, and wheels \cite{mrcrWheels}, and there is a lower bound for $n$-dimensional hypercubes which is conjectured to be exact \cite{maxcrossingcube}.  The maximum rectilinear crossing number has also been computed for $C_5 \times C_5$ \cite{C5xC5},  the Petersen graph \cite{mrcrPetersen}, and $Q_3$, the $3$-dimensional hypercube \cite{maxcrossingcube}.  For some other lesser known families of graphs, see \cite{mrcrGameBoards,mrcrPolynomio}.  In another direction, some work has been done to determine the maximum and minimum values of the maximum rectilinear crossing number in a given family \cite{mrcrPQgraphs,mrcrMaxMinChordedCycles,mrcrMaxMin2reg,mrcrMaxMinSmallCubic,mrcrMaxMinRegular}.





\section{A Lower Bound}
\label{sec:LB}


Let $S$ be a spider  with $k\geq 3$ legs $L_1,L_2,\dots,L_k$.  Let $\ell_i$ denote the number of edges in leg $L_i$, with $\ell_1 \geq \ell_2 \geq \cdots \geq \ell_k \geq 2$. In this section we give an algorithm for constructing a rectilinear drawing of $S$ and use it to find a lower bound on $\mrcr\lp S\rp$.

We begin by labeling the vertices of $S$ as follows. Label the degree-$k$ vertex $(0,0)$. Assign the label $(i,j)$ to the vertex on $L_i$ at distance $j$ from $(0,0)$. We place the vertices of $S$ anticlockwise on a circle according to the cyclic order $[V_0,V_1,V_2,V_3,V_4]$, where the vertices in each subsequence $V_i$ are ordered as follows: 
\begin{itemize}
\item $V_0$: $(0,0)$
\item $V_1$: vertices with $i$ odd and $j$ even, sorted by decreasing $i$ then increasing $j$
\item $V_2$: vertices with $i$ even and $j$ odd, sorted by increasing $i$ then decreasing $j$
\item $V_3$: vertices with both $i$ and $j$ odd, sorted by decreasing $i$ then increasing $j$
\item $V_4$: vertices with both $i$ and $j$ even and positive, sorted by increasing $i$ then decreasing $j$
\end{itemize}
After placing all vertices we draw the edges of $S$ as straight-line segments. If necessary, we perturb the locations of the vertices to ensure our drawing is good.

\begin{ex}
In Figure \ref{fig:algorithm} we draw a spider with legs of length $4$, $3$, $2$, and $2$ according to the algorithm. Call this graph $S$. There are $\vartheta(S) - 2 = 40$ crossings.  The two missed crossings are between the edges $(0,0)-(2,1)$ and $(4,1)-(4,2)$ and between the edges $(0,0)-(1,1)$ and $(3,1)-(3,2)$.  Every other pair of non-incident edges crosses.
\begin{figure}[!ht]
\begin{center}
\begin{multicols}{2}
\begin{tikzpicture}[scale=1.2]
\fill (0,0) circle (2pt);
\fill (1,0) circle (2pt);
\fill (2,0) circle (2pt);
\fill (3,0) circle (2pt);
\fill (4,0) circle (2pt);
\draw (0,0) -- (1,0);
\draw (1,0) -- (2,0);
\draw (2,0) -- (3,0);
\draw (3,0) -- (4,0);
\draw (.3,.25) node {\tiny $(0,0)$};
\draw (1.3,.25) node {\tiny $(1,1)$};
\draw (2.3, .25) node {\tiny $(1,2)$};
\draw (3.3, .25) node {\tiny $(1,3)$};
\draw (4.3, .25) node {\tiny $(1,4)$};
\fill (0,1) circle (2pt);
\fill (0,2) circle (2pt);
\fill (0,3) circle (2pt);
\draw (0,0) -- (0,1);
\draw (0,1) -- (0,2);
\draw (0,2) -- (0,3);
\draw (.4,1) node {\tiny $(2,1)$};
\draw (.4, 2) node {\tiny $(2,2)$};
\draw (.4,3) node {\tiny $(2,3)$};
\fill (-1,0) circle (2pt);
\fill (-2,0) circle (2pt);
\draw (0,0) -- (-1,0);
\draw (-1,0) -- (-2,0);
\draw (-1.3, .25) node {\tiny $(3,1)$};
\draw (-2.3, .25) node {\tiny $(3,2)$};
\fill (0,-1) circle (2pt);
\fill (0,-2) circle (2pt);
\draw (0,0) -- (0,-1);
\draw (0,-1) -- (0,-2);
\draw (.4, -1) node {\tiny $(4,1)$};
\draw (.4, -2) node {\tiny $(4,2)$};
\end{tikzpicture}

\begin{tikzpicture}[scale=.8]
\draw (0,0) circle (3cm);
\draw (0,-2.75) -- (0,-3.25);
\draw (0,2.75) -- (0,3.25);
\draw (2.75,0) -- (3.25,0);
\draw (-2.75,0) -- (-3.25,0);
\draw (-90:4.2) node {$V_0$};
\fill (-90:3) circle (2pt);
\draw (-90:3.5) node {$(0,0)$};
\draw (-40:4.5) node {$V_1$};
\fill (-70:3) circle (2pt);
\draw (-70:3.5) node {$(3,2)$};
\fill (-35:3) circle (2pt);
\draw (-35:3.5) node {$(1,2)$};
\fill (-25:3) circle (2pt);
\draw (-25:3.6) node {$(1,4)$};
\draw (35:4.5) node {$V_2$};
\fill (15:3) circle (2pt);
\draw (15:3.6) node {$(2,3)$};
\fill (45:3) circle (2pt);
\draw (45:3.5) node {$(2,1)$};
\fill (60:3) circle (2pt);
\draw (60:3.5) node {$(4,1)$};
\draw (145:4.5) node {$V_3$};
\fill (100:3) circle (2pt);
\draw (100:3.5) node {$(3,1)$};
\fill (128:3) circle (2pt);
\draw (128:3.5) node {$(1,1)$};
\fill (157:3) circle (2pt);
\draw (157:3.6) node {$(1,3)$};
\draw (220:4.5) node {$V_4$};
\fill (200:3) circle (2pt);
\draw (200:3.6) node {$(2,2)$};
\fill (229:3) circle (2pt);
\draw (229:3.5) node {$(4,2)$};
\draw (-90:3) -- (128:3);
\draw (128:3) -- (-35:3);
\draw (-35:3) -- (157:3);
\draw (157:3) -- (-25:3);
\draw (-90:3) -- (45:3);
\draw (45:3) -- (200:3); 
\draw (200:3) -- (15:3);
\draw (-90:3) -- (100:3);
\draw (100:3) -- (-70:3);
\draw (-90:3) -- (60:3);
\draw (60:3) -- (229:3);
\end{tikzpicture}
\end{multicols}
\caption{At left, a spider labeled as described in the algorithm. At right, the same spider drawn as described in the algorithm.}
\label{fig:algorithm}
\end{center}
\end{figure}
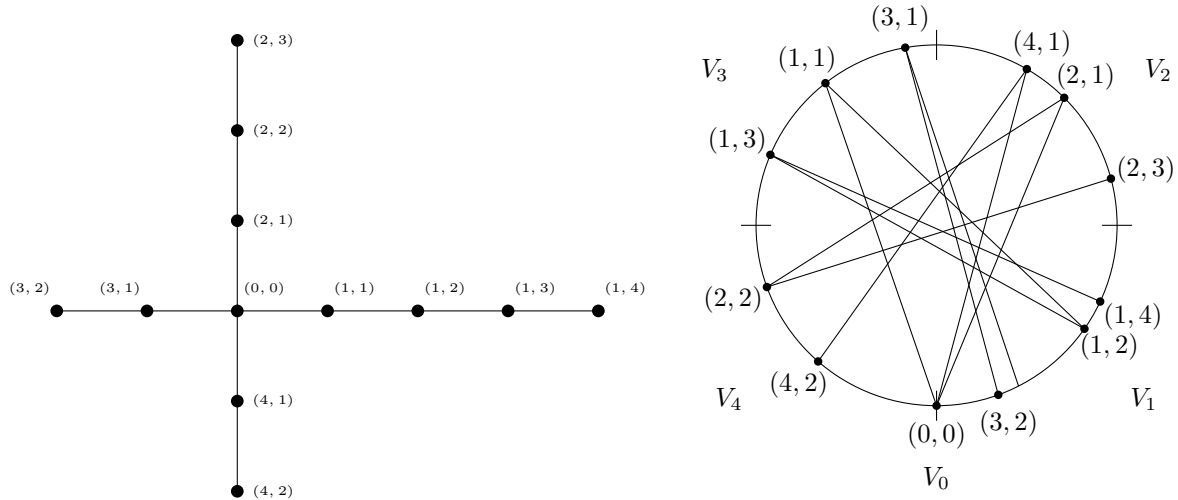
\end{ex}
This drawing algorithm establishes the following lower bound on all spiders.
\begin{prop}
\label{prop:LB}
Let $S$ be a spider with $k\geq 3$ legs of lengths $\ell_1\geq \ell_2\geq \cdots \geq\ell_k$. Then 
\[\displaystyle\mrcr \lp S\rp\geq \vartheta\lp S \rp-\sum_{i=3}^{k}\lp \ell_i-1\rp\left\lfloor\frac{i-1}{2}\right\rfloor.\]
\end{prop}

\begin{proof}
We claim that each pair of nonadjacent edges belonging to the same leg 
cross each other. We note that deleting $V_0$ and labeling $(0,0)$ as $(i,0)$ to place it in $V_1$ or $V_4$ is consistent with the cyclic ordering of vertices of $L_i$. Let $e_1=\lp i,j_1\rp$-$\lp i,j_1+1\rp$ and $e_2=\lp i,j_2\rp$-$\lp i,j_2+1\rp$ be edges of $L_i$ with $j_2>j_1+1$. We note that both $e_1$ and $e_2$ have an end-vertex in each of $V_m$ and $V_{m+2}$ for $m=1$ or $m=2$ (depending on the parity of $i$) and $e_1$'s end-vertices follow (or precede, again depending on the parity of $i$) $e_2$'s in both $V_m$ and $V_{m+2}$ since $j$ either increases in both or decreases in both. In each case, $e_1$ and $e_2$ cross.  For example, see Figure \ref{fig:sameleg}.

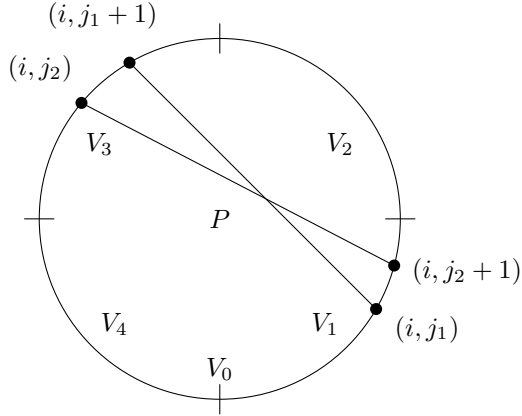
\begin{figure}
\begin{center}
\begin{tikzpicture}[scale=.8]
\draw (0,0) circle (3cm);
\draw (0,-2.75) -- (0,-3.25);
\draw (0,2.75) -- (0,3.25);
\draw (2.75,0) -- (3.25,0);
\draw (-2.75,0) -- (-3.25,0);
\draw (0,-2.5) node {$V_0$};
\draw (1.75,-1.75) node {$V_1$};
\draw  (2,1.25) node {$V_2$};
\draw (-2,1.25) node {$V_3$};
\draw (-1.75,-1.75) node {$V_4$};
\fill (-30:3) circle (.1cm);
\fill (-15:3) circle (.1cm);
\fill (120:3) circle (.1cm);
\fill (140:3) circle (.1cm);
\draw (-30:3) -- (120:3);
\draw (-15:3) -- (140:3);
\draw (-28:3.9) node {$(i,j_1)$};
\draw (-12:4.2) node {$(i,j_2+1)$};
\draw (120:3.9) node {$(i,j_1+1)$};
\draw (140:3.9) node {$(i,j_2)$};
\end{tikzpicture}
\caption{Two edges from the same leg cross when $i$ and $j_2$ are odd and $j_1$ is even.}
\label{fig:sameleg}
\end{center}
\end{figure}

Next we will consider the case where $e_1$ and $e_2$ are nonadjacent edges on different legs. Let $e_1=\lp i_1,j_1\rp$-$\lp i_1,j_1+1\rp$ and $e_2=\lp i_2,j_2\rp$-$\lp i_2,j_2+1\rp$ with $i_1\neq i_2$ . We distinguish three cases:
\begin{case}
	If $i_1$ and $i_2$ have different parity, then at most one of $e_1$ and $e_2$ may have an end vertex $(0,0)$, which we may consider without loss of generality to be the first vertex of $V_1$ or the last vertex of $V_4$.
    In this case, one of $e_1$ or $e_2$ has end vertices in $V_1$ and $V_3$ and the other has end vertices in $V_2$ and $V_4$. Thus, $e_1$ and $e_2$ cross because their end vertices alternate in cyclic order.
\end{case} 
\begin{case}
	If $i_1$ and $i_2$ have the same parity and both $j_1$ and $j_2$ are positive, then, without loss of generality, we assume $i_1 <i_2$.
    Then by construction, the cyclic order of the end vertices is $\lp i_2,odd\rp$-$\lp i_1,even \rp$-$\lp i_2, even\rp$-$\lp i_1,odd\rp$, so $e_1$ and $e_2$ cross. 
\end{case}
\begin{case}
	If  $e_1=\lp 0,0\rp$-$\lp i_1,1\rp$, $i_1$ and $i_2$ have the same parity, and $i_2$ is positive, then we consider two subcases:
    \begin{subcase}
    	If $i_1>i_2$ then: \begin{itemize}
\item $e_1$'s end vertices are in $V_0$ and $V_2$ and $e_2$'s are in $V_2$ (preceding $\lp i_1,1\rp$) and $V_4$.
\item $e_1$'s end vertices are in $V_0$ and $V_3$ and $e_2$'s are in $V_3$ (following $\lp i_1,1\rp$) and $V_1$.
\end{itemize}
Either way, $e_1$ and $e_2$ cross.
    \end{subcase}
    \begin{subcase}
    \label{subcase:nocross}
    	If $i_1<i_2$ then the end vertices of $e_1$ and $e_2$ do not alternate since $i$ increases in $V_2$ and decreases in $V_3$. This means that $e_1$ and $e_2$ do not cross.
    \end{subcase}
\end{case}
The only time that crossings are missed is in Case~\ref{subcase:nocross}.  In particular, we have $\ell_{i_2}-1$ missed crossings between $L_{i_1}$ and $L_{i_2}$. As there are $\left\lfloor\frac{i_2-1}{2}\right\rfloor$ such pairs $i_1,i_2$, each $i>2$ contributes $\lp \ell_i-1\rp\left \lfloor\frac{i-1}{2}\right \rfloor$ missed crossings. 
\end{proof}

For a spider with at least three legs, we have found a convex  drawing, that is, the vertices of the drawing are in convex position.   
This gives us a lower bound on the maximum rectilinear crossing number of spiders.  It was conjectured that every graph has a convex rectilinear drawing maximizing the rectilinear crossing number \cite{ConvexConjecture}. While this has been disproven \cite{ConvexConjectureDisproof}, it is still open for trees.

\section{An Upper Bound}
\label{sec:UB}

Our lower bound on the maximum rectilinear crossing number of a spider was achieved by constructing a good drawing of the spider. The logic here is reversed from what it is for the crossing number, where a specific drawing would show an upper bound. This suggests that obtaining an upper bound on the maximum crossing number (rather like the lower bound on the crossing number) is the tougher problem, since it requires an argument applying to all good drawings of the graph. We work with an auxillary graph, in which non-edges correspond to non-thrackled pairs of spider legs guaranteed by Theorem \ref{thm:S333}. For each non-edge of the auxiliary graph we subtract one from the thrackle bound to give our upper bound. Our result is stated in the following proposition.


\begin{prop}
\label{prop:UB}
	Let $S$ be a spider with $k\geq 3$ legs and assume each leg has length at least 2.  Then 
    \[
    	\mrcr(S) \leq \vartheta(S) - \left(\binom{k}{2} - \left\lfloor\frac{k^2}{4}\right\rfloor\right).
    \]
\end{prop}
\begin{proof}
	Begin with a rectilinear drawing $D$ of $S$. Construct an auxiliary graph $G_S(D)$ on $k$ vertices in which each vertex corresponds to a particular leg of $S$ and each edge corresponds to a pair of legs which are thrackled in $D$.  A pair of legs is thrackled if any two nonadjacent edges in the two legs cross.
    
    We claim $G_S(D)$ is triangle-free.  If not, we would have three pairwise-thrackled legs of $S$ in $D$, which implies that a subgraph of $S$ isomorphic to $S_3^2$ is thrackled in $D$. This contradicts Theorem~\ref{thm:S333}.
    
    Since $G_S(D)$ is triangle-free, Mantel's Theorem \cite{Mantel} implies that $G_S(D)$ contains at most $\left\lfloor\frac{k^2}{4}\right\rfloor$ edges.  The minimum number of nonedges in $G_S(D)$ is thus $\binom{k}{2} - \left\lfloor\frac{k^2}{4}\right\rfloor$.  Since nonedges in $G_S(D)$ correspond to non-thrackled leg pairs, and each non-thrackled leg pair is missing at least one crossing, we get that
    \[
    	\mrcr(S) \leq \vartheta(S) - \left(\binom{k}{2} - \left\lfloor\frac{k^2}{4}\right\rfloor\right),
    \]
    as desired.
\end{proof}

We can now state an exact value for the maximum rectilinear crossing number of any spider in which all legs have length two.

\begin{thm}
\label{thm:Len2}
	Let $S_k^2$ be a spider with $k\geq 3$ legs where each leg has length 2.  Then
    \[
    	\mrcr(S_k^2) = \binom{2k}{2} - 2\binom{k}{2} - k + \left\lfloor\frac{k^2}{4}\right\rfloor.
    \]
\end{thm}
\begin{proof}
	First observe that $S_k^2$ has $n=2k+1$ vertices and $m=2k$ edges.  Additionally, one vertex of $S_k^2$ has degree $k$, $k$ vertices have degree $2$, and $k$ vertices have degree $1$.  Therefore, the thrackle bound for $S_k^2$ is 
    \[
    	\vartheta(S_k^2) = \binom{2k}{2}- \binom{k}{2} - k. \label{eqn:Len2ThrackleBd}
    \]
    According to Proposition~\ref{prop:UB},
    \begin{align*}
    	\mrcr(S_k^2) &\leq \vartheta(S) - \left(\binom{k}{2} - \left\lfloor\frac{k^2}{4}\right\rfloor\right) \\
        &= \binom{2k}{2} - 2\binom{k}{2} - k + \left\lfloor\frac{k^2}{4}\right\rfloor.
    \end{align*}
    Further, according to Proposition~\ref{prop:LB},
    \begin{align*}
    	\mrcr(S_k^2) &\geq \vartheta(S_k^2) - \sum_{i=3}^k\left\lfloor\frac{i-1}{2}\right\rfloor \\
        &= \vartheta(S_k^2) - \begin{cases}2\sum_{i=1}^{(k-2)/2}i & \text{ for $k$ even} \\ \frac{k-1}{2} + 2\sum_{i=1}^{(k-3)/2}i & \text{ for $k$ odd} \end{cases} \\
        &= \vartheta(S_k^2) - \binom{k}{2} + \begin{cases} \frac{k^2}{4} & \text{ for $k$ even} \\ \frac{k^2-1}{4} & \text{ for $k$ odd} \end{cases} \\
        &= \vartheta(S_k^2) - \binom{k}{2} + \left\lfloor\frac{k^2}{4}\right\rfloor \\
        &= \binom{2k}{2} - 2\binom{k}{2} - k + \left\lfloor\frac{k^2}{4}\right\rfloor.
    \end{align*}
    Together, these upper and lower bounds give us the desired equality.
\end{proof}



\section{Future Work}
\label{sec:futurework}

While the bounds given by our drawing algorithm and Mantel's theorem agree when each leg of the spider has length $2$, this method does not seem to generalize.  In our proof of the upper bound, we may only subtract one crossing from the thrackle bound for every triple of legs, but in a spider with longer legs our algorithm misses many more crossings.  We believe that these crossings must be missed, and thus we conjecture that the lower bound given in Proposition~\ref{prop:LB} is correct.~\footnote{Since the time of acceptance, Bennet, English, and Talanda-Fisher have settled this conjecture in the affirmative. They also found (and proved) the maximum rectilinear crossing numbers of trees with diameter $4$.~\cite{ConjectureProof}}

\begin{conj}\label{conj:mrcrspider}
	If $S$ is a spider with $k\geq 3$ legs of lengths $\ell_1\geq \ell_2 \geq \cdots \geq \ell_k$, then
    \[
    	\mrcr(S) = \vartheta(S) - \sum_{i=3}^k (\ell_i -1)\left\lfloor\frac{i-1}{2}\right\rfloor.
    \]
\end{conj}

Further, if our lower bound is correct, then we have exhibited a convex rectilinear drawing which achieves the maximum rectilinear crossing number. This would prove a special case of the conjecture in \cite{ConvexConjecture}. 
Though the conjecture is not true in general \cite{ConvexConjectureDisproof}, we believe that there is a convex drawing obtaining the maximum rectilinear crossing number for spiders and possibly trees in general. 

A difficult problem that is of interest to us is finding the maximum rectilinear crossing number for trees.  To this end, it would be nice to find a proof for Conjecture~\ref{conj:mrcrspider}.  If that can be done, some additional tree families to study are double spiders, which consist of two spider graphs with a single shared leg, and spiders whose legs have been replaced by caterpillars.

As mentioned in the introduction, there is some research done on finding extremal values for the maximum rectilinear crossing number in some graph families. It would be interesting to know which tree was the furthest from thrackleable. That is, what is the minimum value among the maximum rectilinear crossing numbers of trees on $n$ vertices?




\section{Acknowledgements}
This material is based upon work supported by the National Science Foundation under Grant Number DMS 1641020.  Sober\'on's research is also supported by the National Science Foundation Grant Number DMS 1851420. We would like to thank the American Mathematical Society, \'Eva Czabarka, Silvia Fern\'andez-Merchant, Gelasio Salazar, and L\'aszl\'o A. Sz\'ekely for organizing the Mathematics Research Communities workshop ``Beyond Planarity''.

\bibliographystyle{amsplain}
\bibliography{MaxCr}

\end{document}